\documentclass[11pt]{article}
\usepackage[T1]{fontenc}
\usepackage{lmodern}
\usepackage[a4paper,margin=27mm]{geometry}
\usepackage{amsmath,amssymb,amsthm,mathtools,booktabs,array,microtype,needspace}
\usepackage[hypertexnames=false,colorlinks=true,linkcolor=blue,citecolor=blue,urlcolor=blue]{hyperref}
\hypersetup{pdftitle={Higher-order colossally abundant numbers},pdfauthor={Oleg R. Musin}}
\newtheorem{theorem}{Theorem}[section]
\newtheorem{proposition}[theorem]{Proposition}
\newtheorem{lemma}[theorem]{Lemma}
\newtheorem{corollary}[theorem]{Corollary}
\theoremstyle{definition}
\newtheorem{definition}[theorem]{Definition}
\newtheorem{problem}[theorem]{Problem}

\theoremstyle{remark}

\DeclareMathOperator*{\argmax}{arg\,max}
\DeclareMathOperator{\Li}{Li}
\newcommand{\CA}{\mathrm{CA}}
\newcommand{\RH}{\mathrm{RH}}
\newcommand{\Rgap}{\mathcal R}
\newcommand{\Cl}{\mathcal C}
\title{Higher-order colossally abundant numbers}
\author{Oleg R. Musin}
\date{}
\begin{document}
\maketitle

\begin{abstract}
We construct nested subclasses of colossally abundant numbers by
changing the coordinates in their supporting-line description.
The least member $m_k$ of the class of order $k$ is called its
first contact. Our main questions concern the growth of $m_k$,
the factorization of the numbers in each class, and the intersection
of all classes. For coordinates
$(e^{a u^\beta},-(\sigma(n)/n)^k)$, where $u=\log\log n$, every
class is infinite, the intersection is empty, and
$u(m_k)\sim(k/(a\beta))^{1/\beta}$, unconditionally.
Every fixed integer divides all members of sufficiently high
classes; in a specified range of coordinates, $m_k\mid m_{k+1}$.
Robin's inequality restricted to any one of these classes is
equivalent to the Riemann hypothesis (RH). We also construct
families whose intersection is exactly the set of global
maximizers of $\sigma(n)/(n\log\log n)$ on the CA tail;
for these families, $m_k\to\infty$ is equivalent to RH.
Finally, a bound $\Re\varrho\le b<1$ for the nontrivial zeta
zeros gives, for the natural choice of abscissa $\log n$,
$\log\log m_k=k+O(k^3e^{-(1-b)k})$ and estimates for the prime
factors of $m_{k+1}/m_k$. The value $b=7/8$ uses the result
stated in the cited OpenAI preprint and requires no assumption
of RH.
\end{abstract}

\section{Introduction}\label{sec:intro}

A positive integer $N$ is \emph{colossally abundant} (CA) if, for
some $\varepsilon>0$,
\begin{equation}\label{eq:classical}
 \frac{\sigma(N)}{N^{1+\varepsilon}}
 \ge \frac{\sigma(n)}{n^{1+\varepsilon}}
 \qquad(n\ge1),
\end{equation}
where $\sigma(n)=\sum_{d\mid n}d$. All maximizing integers are
included when equality occurs. These numbers arise in the study
of large values of the divisor sum, initiated by Ramanujan
\cite{Ramanujan} and developed by Alaoglu and Erd\H{o}s \cite{AE}.
Throughout the paper, write
\begin{equation}\label{eq:notation}
 E=e^\gamma,\qquad \rho(n)=\frac{\sigma(n)}n,\qquad
 u(n)=\log\log n,\qquad G(n)=\frac{\rho(n)}{u(n)}.
\end{equation}

Taking logarithms in \eqref{eq:classical} shows that CA numbers
are the points at which $-\log\rho(n)+\varepsilon\log n$
attains its minimum. Equivalently, they are contacts of supporting
lines with the lower convex envelope of
$(\log n,-\log\rho(n))$; see \cite[Section 3, Example 1]{Musin}.
We obtain \emph{higher-order colossally abundant numbers} by
changing these coordinates and taking the new contacts.
Increasing concave changes produce nested sets. Different
choices of coordinates give different families, whose arithmetic
and intersections are the subject of this paper.

The simplest example keeps the abscissa $\log n$ and replaces
$-\log\rho(n)$ by $-\rho(n)^k$. Fix the reference integer
$n_*=55440$ and the CA tail $D=\{n\in\CA:n\ge n_*\}$.
For $k\ge1$, let $A_k(1)$ consist of those $N\in D\setminus\{n_*\}$
for which, for some $\lambda>0$,
\begin{equation}\label{eq:natural-intro}
 \rho(N)^k-\lambda\log N
 =\max_{n\in D}\bigl(\rho(n)^k-\lambda\log n\bigr).
\end{equation}
The reference remains in the comparison even though it is
excluded from the resulting class. The sets $A_k(1)$ are nested
and infinite, but their intersection is empty. Their least
members, or \emph{first contacts}, satisfy
\[
 \log\log m_k\sim k,\qquad m_k\mid m_{k+1}.
\]
Both conclusions are unconditional. For the abscissa
$F(u)=e^{a u^\beta}$, we prove
$u(m_k)\sim(k/(a\beta))^{1/\beta}$.
The choice $F(u)=e^{a(\log u)^\eta}$, with $\eta>1$, gives
$\log u(m_k)\sim(k/(a\eta))^{1/(\eta-1)}$.
We call these the \emph{arithmetic families}. Every fixed integer
divides every member of all sufficiently high levels of each
such family. Under a concavity condition on the abscissa, we
also obtain explicit bounds on the order required for divisibility
and restrictions on the individual prime exponents.

A second question is whether a smaller class of CA numbers still
suffices for testing Robin's inequality. Robin's theorem
\cite{Robin} states that
\begin{equation}\label{eq:Robin}
 \RH\quad\Longleftrightarrow\quad
 \sigma(n)<e^\gamma n\log\log n\quad(n>5040),
\end{equation}
and it suffices to check this inequality on CA numbers.
We prove that it also suffices to check it on any fixed level
of an arithmetic family. The proof transfers a counterexample
to a supporting contact by using the concavity of the transformed
boundary $\rho=Eu$.

There is another construction with a different intersection.
In coordinates $(-q_k,-G^{s_k})$, where $q_k(n)>0$ decreases
with $n$, a contact maximizes $G(n)^{s_k}+\lambda q_k(n)$.
Every global maximizer of $G$ on $D$ is therefore retained at
every level. For suitable sequences $q_k$ and $s_k$, we prove
that the intersection consists of exactly these global maximizers.
Since $G(n)\to E$ on the CA numbers, Robin's theorem implies
that this intersection is empty precisely when RH holds.
Thus, with $\min\varnothing=+\infty$, divergence of the first
contacts is equivalent to RH for these
\emph{maximum-preserving families}. In the arithmetic families
the first contacts diverge unconditionally. These two cases show
why the intersection must be studied separately from the
reduction of Robin's criterion to each level.

The last section connects the arithmetic families directly with
the location of the nontrivial zeros of $\zeta(s)$. If all these
zeros satisfy $\Re\varrho\le b<1$, then for $A_k(1)$ we obtain
\[
 \log\log m_k=k+O(k^3e^{-(1-b)k}).
\]
We also estimate $G$ uniformly on the whole level and count the
prime factors of the integer $m_{k+1}/m_k$. These results use
the classical CA parameters and estimates for the distribution
of primes. Specializing to $b=7/8$ uses the theorem stated in
the OpenAI preprint \cite{OpenAI}; its status and its role as an
external input are specified in Section~\ref{sec:zero-free}.

Section~\ref{sec:construction} gives the general construction.
Section~\ref{sec:arithmetic} proves the growth and factorization
results for the arithmetic families.
Section~\ref{sec:maxima} determines the intersections of the
maximum-preserving families. Section~\ref{sec:examples} gives
numerical examples, applies the Ramanujan--Nicolas estimates,
and compares the construction with Wu's $\CA_2$.
Section~\ref{sec:zero-free} proves the quantitative consequences
of a zero-free half-plane.

\section{Coordinates, nesting, and first contacts}\label{sec:construction}

\subsection{Classical facts and the comparison set}
We fix
\begin{equation}\label{eq:D}
 n_*=55440,\qquad D=\{n\in\CA:n\ge n_*\},\qquad
 u_*=u(n_*),\quad \rho_*=\rho(n_*),\quad G_*=G(n_*).
\end{equation}
The integer $n_*$ is the first CA number above $5040$, and direct
evaluation gives $G_*<E$. It will always remain a competitor, but
will be excluded from the reported higher-order classes. Robin's
CA reduction gives
\begin{equation}\label{eq:RobinD}
 \RH\quad\Longleftrightarrow\quad G(n)<E\quad(n\in D).
\end{equation}
We use the unconditional estimate \cite{Robin}
\begin{equation}\label{eq:Robinupper}
 \rho(n)\le Eu(n)+\frac{0.6483}{u(n)}\qquad(n\ge3).
\end{equation}
If RH is false, there are arbitrarily large CA numbers with
$G(n)>E$. Indeed, Robin's oscillation theorem gives arbitrarily
large strict counterexamples, and his reduction between
consecutive CA numbers transfers these to unbounded CA numbers.
The geometric reduction follows by placing $\log\rho$ below its
upper concave envelope at $x=\log n$: the boundary
$\log(E\log x)$ is concave, so a counterexample forces a
counterexample at an endpoint of its CA interval.

We shall also use the classical prime-exponent description. For a prime $p$ and $j\ge1$,
put
\begin{equation}\label{eq:critical}
 S_j(p)=p+p^2+\cdots+p^j,\qquad
 \Phi(p,j)=\frac{\log(1+1/S_j(p))}{\log p}.
\end{equation}
For a fixed $\varepsilon$, the exponent of $p$ in a maximizer of
\eqref{eq:classical} includes every increment for which
$\Phi(p,j)>\varepsilon$ and may include an increment for which
$\Phi(p,j)=\varepsilon$. Decreasing $\varepsilon$ therefore generates
the CA numbers by successive prime-exponent changes. This is the
prime-by-prime description of Alaoglu and Erd\H{o}s
\cite[Theorem 10]{AE}; the formulation above explicitly retains
ties. It does not require a conjecture about coincident critical
values. The first nontrivial CA numbers are
\[
 2,\ 6,\ 12,\ 60,\ 120,\ 360,\ 2520,\ 5040,\ 55440,\ 720720,\ldots.
\]
Every CA number is superabundant: its ratio $\sigma(n)/n$ exceeds
that of every smaller positive integer. Its prime factors form an
initial segment of the primes, their exponents are nonincreasing,
and its largest prime factor $P(n)$ satisfies $P(n)\sim\log n$.
For factorization and distribution results see \cite{AE,EN,CNS};
accessible accounts are given in \cite{Lagarias,NS}.

We record the classical size information in the form needed below.
Here $v_p(n)$ is the exponent of $p$ in $n$, and
$\vartheta(x)=\sum_{p\le x}\log p$.

\begin{lemma}\label{lem:CA-scale}
As $n\to\infty$ through CA numbers,
\begin{equation}\label{eq:CAscale}
 P(n)\sim\log n,\qquad \rho(n)\sim E u(n),\qquad G(n)\to E.
\end{equation}
For every fixed prime $p$,
\begin{equation}\label{eq:fixedprime}
 v_p(n)\log p=u(n)+\log u(n)+O_p(1).
\end{equation}
For every real $v\to\infty$ there is a CA number $n_v$ with
$u(n_v)=v+o(1)$.
\end{lemma}
\begin{proof}
For a maximizing parameter $\varepsilon$, write
$\varepsilon=\Phi(z,1)$ with $z$ real. The prime-increment rule gives
$P(n)\le z\le p^+(n)$, with endpoints allowed; here $p^+(n)$ is
the next prime after $P(n)$. The standard CA estimate
\cite[Lemma 3]{CNS} is
\[
 \vartheta(P(n))\le\log n\le\vartheta(z)+O(\sqrt z).
\]
Together with the prime number theorem, it gives $P(n)\sim\log n$.
Also $\varepsilon\sim1/(P(n)\log P(n))$. For fixed $p$, the
inequalities
$\Phi(p,v_p(n)+1)\le\varepsilon\le \Phi(p,v_p(n))$ and
$\Phi(p,j)\asymp_p p^{-j}$ give \eqref{eq:fixedprime}.

In the Euler product
\[
 \rho(n)=\prod_{p\le P(n)}(1-p^{-v_p(n)-1})
                  \prod_{p\le P(n)}(1-p^{-1})^{-1},
\]
the first product tends to $1$: each fixed exponent tends to
infinity, while the logarithms of the factors are dominated by
a summable constant multiple of $p^{-2}$. Mertens' product theorem
then gives $\rho(n)\sim E\log P(n)\sim Eu(n)$.
Finally choose a prime $p\sim e^v$ and the largest maximizer at
$\varepsilon=\Phi(p,1)$. Its largest prime factor is $p$, so
$\log n\sim p$ and $u(n)=v+o(1)$.
\end{proof}

For completeness, the increment rule above follows
from the exact identity
\[
 \frac{\rho(p^j)p^{-j\varepsilon}}
      {\rho(p^{j-1})p^{-(j-1)\varepsilon}}
 =(1+1/S_j(p))p^{-\varepsilon}.
\]
The increments decrease strictly with $j$, and multiplicativity
permits independent optimization at each prime. At a noncritical
parameter this also gives the familiar formula
\[
 v_p(n_\varepsilon)=
 \left\lfloor\frac{\log(p^{1+\varepsilon}-1)
                   -\log(p^\varepsilon-1)}{\log p}\right\rfloor-1.
\]
At critical parameters the increment rule, rather than a convention
for this floor, specifies all ties.

\subsection{The contact construction}
For coordinates $T=(X,Y)$ on $D$, define
\begin{equation}\label{eq:contact}
 \Cl_D(X,Y)=\left\{n\in D\setminus\{n_*\}:
 \begin{array}{l}
 \text{there is a }\lambda\ge0\text{ such that}\\[-2pt]
 Y(n)+\lambda X(n)\le Y(m)+\lambda X(m)\quad(m\in D)
 \end{array}\right\}.
\end{equation}
A coordinate scheme $T=(T_k)_{k\ge1}$ defines
$C^{(k)}(T)=\Cl_D(X_k,Y_k)$. Horizontal supports and all ties are
retained. Only attained contacts count; a limiting point of the hull
does not represent an integer. The definition always compares with
the same set $D$.

\begin{theorem}[Concave changes of coordinates]\label{thm:nesting}
Suppose
\[
 X_{k+1}=\phi_k\circ X_k,\qquad
 Y_{k+1}=\psi_k\circ Y_k,
\]
where $\phi_k,\psi_k$ are increasing concave differentiable
functions with positive derivatives on intervals containing the
respective coordinate ranges. Then
$C^{(k+1)}(T)\subseteq C^{(k)}(T)$.
\end{theorem}
\begin{proof}
Let $n$ have support parameter $\lambda$ at level $k+1$.
Concavity gives, for every $m\in D$,
\begin{align*}
 0&\le \psi_k(Y_k(m))-\psi_k(Y_k(n))
       +\lambda[\phi_k(X_k(m))-\phi_k(X_k(n))]\\
 &\le\psi_k'(Y_k(n))[Y_k(m)-Y_k(n)]
       +\lambda\phi_k'(X_k(n))[X_k(m)-X_k(n)].
\end{align*}
Division by $\psi_k'(Y_k(n))>0$ gives a level-$k$ support.
\end{proof}

This theorem proves nesting, without claiming that a finite hull
computed at one level contains all information needed for the next.
A recursive computation must justify its treatment of the reference,
the uncomputed tail, and any limiting support.

\begin{definition}\label{def:classes}
A nested family $(C^{(k)})$ is \emph{Robin-preserving at each level}
if, for every $k\ge1$,
\begin{equation}\label{eq:levelRobin}
 [G(n)<E\text{ for all }n\in C^{(k)}]
 \quad\Longleftrightarrow\quad
 [G(n)<E\text{ for all }n\in D].
\end{equation}
It is \emph{maximum-preserving} if it contains, at every level,
the set $\mathcal M=\argmax_D G$, interpreted as empty when the
supremum is not attained. Equivalently, it preserves the global
minima of the ordinate $-G$. It has the \emph{exact-intersection
property} if $\bigcap_k C^{(k)}=\mathcal M$.
\end{definition}

Maximum preservation implies \eqref{eq:levelRobin} in the present
setting. Under $\neg\RH$, Lemma~\ref{lem:CA-scale} and strict
counterexamples imply that $G$ attains a maximum greater than $E$;
under RH, \eqref{eq:levelRobin} is immediate. The converse need
not hold: the arithmetic examples below are Robin-preserving with
empty intersection even under $\neg\RH$.

\begin{proposition}[First-contact sequences]\label{prop:minima}
For any nested sets $C^{(k)}\subseteq\mathbb N$, put
$m_k=\min C^{(k)}$, with $\min\varnothing=\infty$. Then $(m_k)$
is nondecreasing and
\begin{equation}\label{eq:mincriterion}
 \bigcap_{k\ge1}C^{(k)}=\varnothing
 \quad\Longleftrightarrow\quad m_k\longrightarrow\infty.
\end{equation}
If the sequence is bounded, it is eventually constant, with value
$\min\bigcap_k C^{(k)}$.
\end{proposition}
\begin{proof}
Nesting gives monotonicity. A bounded nondecreasing integer sequence
is eventually constant, say $m_k=m$. Then $m$ belongs to every
level. Conversely, a member of the intersection bounds every $m_k$.
The assertion about the least member follows from $m_k\le n$ for
every $n$ in the intersection.
\end{proof}

Even if every level is infinite, the intersection may be finite
or empty. For instance, $\{1\}\cup\{k+1,k+2,\ldots\}$ has
intersection $\{1\}$. We distinguish \emph{escape through nonempty
levels}, where every $m_k$ is finite, from \emph{termination}, where
some $m_k=\infty$.

\begin{problem}[The first-contact problem]\label{prob:first}
For a given nested family of integer sets, determine whether its
first-contact sequence stabilizes, escapes through nonempty levels,
or terminates. In the escape case, determine its growth and the
arithmetic restrictions that hold uniformly at high levels.
\end{problem}

\section{Arithmetic families with empty intersection}\label{sec:arithmetic}

\subsection{Transfer of Robin counterexamples}
The next theorem separates the Robin reduction from the nesting
and intersection questions.

\begin{theorem}[Concave-boundary transfer]\label{thm:transfer}
Let $F,H$ be positive, increasing, unbounded $C^2$ functions,
with positive first derivatives on the relevant ranges. Suppose
\begin{equation}\label{eq:transferA}
 \frac{H(Eu+0.6483/u)}{F(u)}\longrightarrow0,
\end{equation}
and, for all sufficiently large $u$,
\begin{equation}\label{eq:transferB}
 \frac{F''(u)}{F'(u)}
 \ge E\frac{H''(Eu)}{H'(Eu)}.
\end{equation}
Then $\Cl_D(F(u),-H(\rho))$ is infinite and is a Robin test set:
Robin's inequality on this set is equivalent to Robin's inequality
on $D$. If RH is false, this contact set contains infinitely many
strict counterexamples.
\end{theorem}
\begin{proof}
For every $\lambda>0$, \eqref{eq:Robinupper} and
\eqref{eq:transferA} give
\[
 H(\rho(n))-\lambda F(u(n))\longrightarrow-\infty
 \qquad(n\to\infty,\ n\in D).
\]
The maximum is therefore attained. As $\lambda\downarrow0$,
its maximizers escape every finite set, because $H(\rho(n))$ is
unbounded by Lemma~\ref{lem:CA-scale}. This proves infinitude.

Put $z=F(u)$ and $B(z)=H(EF^{-1}(z))$. Condition
\eqref{eq:transferB} says exactly that $B''(z)\le0$ eventually.
Moreover, $B$ is increasing and unbounded, and $B(z)/z\to0$.
Its tangent intercept
\[
 I(z)=B(z)-zB'(z)
\]
increases eventually to infinity. Indeed, $I'=-zB''\ge0$; if
$I$ were bounded above by $M$, integration of
$(B(z)/z)'=-I(z)/z^2$, using $B(z)/z\to0$, would imply
$B(z)\le M$, a contradiction. In particular, $B'(z)\to0$.

If RH is false, choose CA counterexamples $n_j\to\infty$ with
$\rho(n_j)>Eu(n_j)$, and set
$z_j=F(u(n_j))$, $\lambda_j=B'(z_j)>0$. At $n_j$ the objective
$H(\rho)-\lambda_j F(u)$ is strictly greater than $I(z_j)$.
Choose a maximizer $N_j$. Its objective tends to infinity, so
$N_j$ escapes every fixed finite set. For large $j$ it lies in
the concavity range. The tangent inequality
$B(z)\le B(z_j)+\lambda_j(z-z_j)$ then shows that
$H(\rho(N_j))>B(F(u(N_j)))$. Hence $G(N_j)>E$.
These escaping maximizers give infinitely many counterexamples in
the contact set. The implication under RH is immediate.
\end{proof}

\subsection{Two arithmetic hierarchies}
For $a,\beta>0$ define
\begin{equation}\label{eq:Afamily}
 A_k(a,\beta)=\Cl_D(e^{a u^\beta},-\rho^k).
\end{equation}
The original arithmetic family is
$A_k(\alpha)=A_k(\alpha,1)$, whose first coordinate is
$(\log n)^\alpha$. A second family, for $a>0$ and $\eta>1$, is
\begin{equation}\label{eq:Qfamily}
 Q_k(a,\eta)=\Cl_D(e^{a(\log u)^\eta},-\rho^k).
\end{equation}
The range $u\ge u_*>2$ makes both definitions unambiguous.

\begin{theorem}\label{thm:arith-structure}
Each family in \eqref{eq:Afamily} and \eqref{eq:Qfamily} is nested,
has infinite levels, and has empty intersection, unconditionally.
Every level is Robin-preserving. More generally,
\begin{equation}\label{eq:jointnest}
 A_\ell(a',\beta')\subseteq A_k(a,\beta)
 \quad\text{if }\ell\ge k,\ 0<a'\le a,\ 0<\beta'\le\beta.
\end{equation}
\end{theorem}
\begin{proof}
Increasing the exponent of $\rho$ changes the negative ordinate
by $y\mapsto-(-y)^{\ell/k}$, an increasing concave map for
$y<0$. Theorem~\ref{thm:nesting} proves nesting in the order.

For $H(t)=t^k$, condition \eqref{eq:transferA} holds for both
abscissas. For $F(u)=e^{a u^\beta}$, condition
\eqref{eq:transferB} becomes
\[
 a\beta u^\beta\ge k-\beta,
\]
which holds eventually. For $F(u)=e^{a(\log u)^\eta}$, writing
$v=\log u$, it becomes
\[
 a\eta v^{\eta-1}+\frac{\eta-1}{v}\ge k.
\]
Theorem~\ref{thm:transfer} proves infinitude and the Robin property.

Here is a common proof of empty intersection for any fixed
increasing abscissa $F$. Fix $n\in D\setminus\{n_*\}$ and choose
$b\in D$ with $b>n$. Since CA numbers are superabundant,
$\rho_*<\rho(n)<\rho(b)$. With
\[
 \theta=\frac{F(u(n))-F(u_*)}{F(u(b))-F(u_*)}\in(0,1),
\]
we have
\[
 \frac{\rho(n)^k-\rho_*^k}{\rho(b)^k-\rho_*^k}\longrightarrow0.
\]
For sufficiently large $k$ this ratio is less than $\theta$.
The point $(F(u(n)),-\rho(n)^k)$ then lies strictly above the
chord joining the reference and $b$, so it admits no lower
support. Every fixed $n$ is eventually excluded.

Finally put $F_0=e^{a u^\beta}$, $F_1=e^{a'u^{\beta'}}$.
The transition $F_1\circ F_0^{-1}$ is increasing and concave,
because
\[
 \frac{d}{du}\log\frac{F_1'(u)}{F_0'(u)}
 =a'\beta'u^{\beta'-1}-a\beta u^{\beta-1}
     +\frac{\beta'-\beta}{u}\le0\qquad(u\ge u_*>1).
\]
Applying Theorem~\ref{thm:nesting} also to this change proves
\eqref{eq:jointnest}.
\end{proof}

Thus, if Robin's inequality holds on any one of these levels,
it holds on all CA numbers above $5040$ and RH follows. Empty
intersection has a different meaning: a counterexample, if one
exists, can be replaced by ever larger counterexamples at successive
levels. No fixed counterexample has to survive.

\subsection{The first-contact growth laws}
For either family write $m_k(F)=\min\Cl_D(F(u),-\rho^k)$.

\begin{lemma}[Reference secants]\label{lem:secant}
For the abscissas in \eqref{eq:Afamily} and \eqref{eq:Qfamily}, set
\begin{equation}\label{eq:secant}
 J_k(n)=\frac{\rho(n)^k-\rho_*^k}{F(u(n))-F(u_*)},\qquad
 \Lambda_k=\max_{n\in D\setminus\{n_*\}}J_k(n).
\end{equation}
Then $\Lambda_k$ is attained, and $m_k(F)$ is the smallest
maximizer of $J_k$.
\end{lemma}
\begin{proof}
The secants are positive and tend to zero at infinity, so their
maximum is attained. If $b$ is the smallest maximizer, then the
line of slope $-\Lambda_k$ through the reference supports $b$.
For $n_*<n<b$, a parameter $\lambda\ge\Lambda_k$ makes the
reference strictly better than $n$. If $\lambda<\Lambda_k$,
the point $b$ is strictly better than $n$, since
\begin{align*}
 &(\rho(b)^k-\lambda F(u(b)))-(\rho(n)^k-\lambda F(u(n)))\\
 &\qquad=(\Lambda_k-\lambda)(F(u(b))-F(u(n)))\\
 &\qquad\quad +(\Lambda_k-J_k(n))(F(u(n))-F(u_*))>0.
\end{align*}
Thus no earlier point is a contact.
\end{proof}

\begin{theorem}[Escape rates]\label{thm:rates}
Let $m_k^A=\min A_k(a,\beta)$ and $m_k^Q=\min Q_k(a,\eta)$.
Unconditionally,
\begin{align}
 u(m_k^A)&\sim U_k:=\left(\frac{k}{a\beta}\right)^{1/\beta},
       \label{eq:Arate}\\
 \log u(m_k^Q)&\sim V_k:=
       \left(\frac{k}{a\eta}\right)^{1/(\eta-1)}.
       \label{eq:Qrate}
\end{align}
For the first family, the maximum reference secant satisfies
\begin{equation}\label{eq:Lambda}
 \Lambda_k^{1/k}\sim E\left(\frac{k}{a\beta e}\right)^{1/\beta}.
\end{equation}
If $K_A(N)=\max\{k:m_k^A\le N\}$, then
\begin{equation}\label{eq:inverse}
 K_A(N)\sim a\beta(\log\log N)^\beta.
\end{equation}
\end{theorem}
\begin{proof}
The minima tend to infinity by Theorem~\ref{thm:arith-structure}
and Proposition~\ref{prop:minima}. For any sequence of CA numbers
with $u=u(n)\to\infty$, Lemma~\ref{lem:CA-scale} gives, for
$F=e^f$ and $k\to\infty$,
\begin{equation}\label{eq:logsecant}
 \frac1k\log J_k(n)=\log E+\log u-\frac{f(u)}k+o(1).
\end{equation}
The reference terms vanish in both numerator and denominator;
also $\log(\rho(n)/(Eu))=o(1)$.

For $f(u)=a u^\beta$, choose a CA competitor with
$u=U_k+o(1)$ using Lemma~\ref{lem:CA-scale}. Its logarithmic
secant divided by $k$ is
$\log(EU_k)-1/\beta+o(1)$. Comparing the maximizing first
contact with this competitor and putting $t_k=u(m_k^A)/U_k$
gives
\[
 \log t_k-\frac{t_k^\beta}{\beta}\ge-\frac1\beta+o(1).
\]
The expression on the left has a unique maximum at $t=1$ and
tends to $-\infty$ at both ends of $(0,\infty)$. Hence
$t_k\to1$, proving \eqref{eq:Arate}; substitution in
\eqref{eq:logsecant} proves \eqref{eq:Lambda}. Monotonicity of
$m_k^A$ and inversion of \eqref{eq:Arate} give \eqref{eq:inverse}.

For $f(u)=a(\log u)^\eta$, choose a CA competitor with
$u=e^{V_k}+o(1)$. Write $v_k=\log u(m_k^Q)$ and
$t_k=v_k/V_k$. Comparison of logarithmic secants gives
\[
 k v_k-a v_k^\eta
 \ge kV_k-aV_k^\eta+o(k).
\]
Since $k=a\eta V_k^{\eta-1}$, division by $aV_k^\eta$ yields
\[
 \eta t_k-t_k^\eta\ge\eta-1+o(1).
\]
Again the function has its unique maximum at $1$, which proves
\eqref{eq:Qrate}.
\end{proof}

For example, $F=e^{\sqrt u}$ gives $u(m_k)\sim4k^2$,
$F=e^u$ gives $u(m_k)\sim k$, and $F=e^{u^2}$ gives
$u(m_k)\sim\sqrt{k/2}$. The logarithmic choice
$F=e^{(\log u)^2}$ gives $\log u(m_k)\sim k/2$.
The first-contact problem thus distinguishes double-exponential
and triple-exponential scales in the original integer variable.

\subsection{Arithmetic restrictions at high order}

\begin{corollary}\label{cor:div-uniform}
For every fixed positive integer $d$, all members of all sufficiently
high levels of either arithmetic family are divisible by $d$.
Moreover, for each fixed prime $p$,
\begin{align*}
 \log P(m_k^A)&\sim U_k,&
 v_p(m_k^A)&\sim\frac{U_k}{\log p},\\
 \log\log P(m_k^Q)&\sim V_k,&
 \log v_p(m_k^Q)&\sim V_k.
\end{align*}
\end{corollary}
\begin{proof}
Equation~\eqref{eq:fixedprime} shows that $v_p(n)\to\infty$
along all CA numbers. Since every member of level $k$ is at
least $m_k\to\infty$, this holds uniformly over that level.
Apply it to the finitely many prime divisors of $d$. The stated
asymptotics follow from \eqref{eq:CAscale}, \eqref{eq:fixedprime},
and Theorem~\ref{thm:rates}.
\end{proof}

There are sharper conclusions when the abscissa is concave as a
function of $\log n$.

\begin{theorem}[Prime-parameter windows]\label{thm:windows}
Let $F$ be an abscissa of one of the arithmetic families, and
suppose $f(x)=F(\log x)$ is increasing and concave for
$x\ge\log n_*$. If $n\in\Cl_D(F(u),-\rho^k)$, there is
$\varepsilon>0$ such that, for every prime $p$, with $b=v_p(n)$,
\begin{equation}\label{eq:windowlower}
 L_k(p,b):=\frac{(1+1/S_{b+1}(p))^k-1}{k\log p}
 \le\varepsilon,
\end{equation}
and, for $p\mid n$,
\begin{equation}\label{eq:windowupper}
 \varepsilon\le
 U_k(p,b):=\frac{1-(1+1/S_b(p))^{-k}}{k\log p}.
\end{equation}
The number $n$ is the unique classical CA maximizer at this
$\varepsilon$. For each fixed $(p,b)$, $L_k(p,b)$ increases
strictly and $U_k(p,b)$ decreases strictly as $k$ increases.
\end{theorem}
\begin{proof}
First, contacts computed over $D$ are also maximizers of
$\rho(n)^k-\lambda F(u(n))$ over all integers $n\ge n_*$.
To see this, use the classical lower hull in the coordinates
$x=\log n$, $y=-\log\rho(n)$. Every non-CA point lies above
or on a chord between consecutive CA contacts. The maps
$x\mapsto f(x)$ and $y\mapsto-e^{-ky}$ are increasing and
concave, with the latter strictly concave. The transformed point
cannot lie below the corresponding transformed chord: equivalently,
for every $\lambda\ge0$, its value of
$-e^{-ky}+\lambda f(x)$ is greater than the convex combination
of the endpoint values. The ordinate endpoints are distinct
because CA numbers are superabundant. Thus a non-CA integer
cannot improve on the CA maxima. The tail tends to $-\infty$
for the maximized objective when $\lambda>0$, and such a
parameter is necessary since $\rho$ is unbounded.

Fix a supporting $\lambda>0$, put $x=\log n$, and set
$\varepsilon=\lambda f'(x)/(k\rho(n)^k)$. Comparison with $np$
and concavity of $f$ give
\[
 \rho(np)^k-\rho(n)^k
 \le\lambda[f(x+\log p)-f(x)]
 \le\lambda f'(x)\log p,
\]
which is \eqref{eq:windowlower}. Comparison with $n/p$ similarly
gives \eqref{eq:windowupper}. These removals stay in the range
$n/p\ge n_*$: if $P(n)\ge19$, then
$n/P(n)\ge2\cdot3\cdot5\cdot7\cdot11\cdot13\cdot17>n_*$;
if $P(n)\le17$, only $n<17n_*$ needs checking, and the sole CA
number in $(n_*,17n_*)$ is $720720$, with $720720/13=n_*$.

The strict elementary inequalities $e^t>1+t$ and $1-e^{-t}<t$
show that
\[
 \Phi(p,b+1)<L_k(p,b)\le\varepsilon,
 \qquad \varepsilon\le U_k(p,b)<\Phi(p,b)\quad(p\mid n).
\]
Thus every prime exponent is the unique optimum in the classical
increment rule. Finally $(e^{kt}-1)/k$ increases strictly with
$k$, whereas $(1-e^{-kt})/k$ decreases strictly for $t>0$.
\end{proof}

\begin{corollary}[Effective divisibility and divisibility chains]
\label{cor:effective}
Under the hypotheses of Theorem~\ref{thm:windows}, let
$d=\prod_{p\mid d}p^{b_p}>1$ and put
\begin{equation}\label{eq:kd}
 k(d)=\max_{p\mid d}
 \left\lceil
 \frac{\log(1+\log p/\log2)}{\log(1+1/S_{b_p}(p))}
 \right\rceil.
\end{equation}
Every member of every level $k\ge k(d)$ is divisible by $d$.
The first contacts satisfy $m_k(F)\mid m_{k+1}(F)$.
\end{corollary}
\begin{proof}
Every member is even, and \eqref{eq:windowupper} at $p=2$ gives
$\varepsilon<1/(k\log2)$. If $v_p(n)<b_p$, then
\eqref{eq:windowlower} implies
\[
 (1+1/S_{b_p}(p))^k-1<\frac{\log p}{\log2},
\]
contradicting $k\ge k(d)$. For the second claim, the first
contacts are nondecreasing and each is a unique classical
maximizer. If two such maximizers $n<m$ have parameters
$\varepsilon_n,\varepsilon_m$, comparison in
\eqref{eq:classical} gives $\varepsilon_n>\varepsilon_m$.
The prime-increment description then gives $n\mid m$.
\end{proof}

For $F=e^{a u^\beta}$ a sufficient concavity range is
\begin{equation}\label{eq:concaverange}
 0<\beta\le1,\qquad a\beta u_*^{\beta-1}\le1.
\end{equation}
Indeed, the sign of $f''$ is the sign of
$a\beta u^{\beta-1}+(\beta-1)/u-1$. This includes
$F=e^{\sqrt u}$, $e^{2\sqrt u}$, and $e^{\alpha u}$ for
$0<\alpha\le1$. It also includes a logarithmic example outside
this parametrization: for $F=e^{(\log u)^2}$ the sign condition is
\[
 \frac{2\log u-1+1/\log u}{u}\le1.
\]
For $u\ge u_*>2$, its left side is at most
$2/e+(1/\log2-1)/2<1$. No concavity assumption is needed for
Theorems~\ref{thm:arith-structure} and \ref{thm:rates}.

\section{Families preserving global maxima}\label{sec:maxima}

\subsection{The intersection theorem}
We now use coordinates $(-q_k,-G^{s_k})$, so contacts maximize
$G^{s_k}+\lambda q_k$. Let $q_k(n)>0$ decrease strictly to zero
as $n\to\infty$ in $D$, and let $s_k>0$ be nondecreasing.
Assume that the changes from $-q_k$ to $-q_{k+1}$ are increasing
concave differentiable maps with positive derivatives, and that
\begin{equation}\label{eq:weightescape}
 \frac{q_k(n)}{q_k(n_*)}\longrightarrow0
 \quad(k\to\infty)\qquad\text{for every fixed }n>n_*.
\end{equation}
Set
\begin{equation}\label{eq:Bfamily}
 B_k=\Cl_D(-q_k,-G^{s_k}),\qquad
 \mathcal M=\argmax_{n\in D}G(n).
\end{equation}

\begin{theorem}[Exact preservation]\label{thm:exact}
Under these assumptions the sets $B_k$ are nested, are
Robin-preserving at every level, and satisfy
\begin{equation}\label{eq:exact}
 \bigcap_{k\ge1}B_k=\mathcal M.
\end{equation}
Consequently,
\begin{equation}\label{eq:RHequivalences}
 \RH\quad\Longleftrightarrow\quad
 \mathcal M=\varnothing\quad\Longleftrightarrow\quad
 \bigcap_k B_k=\varnothing\quad\Longleftrightarrow\quad
 m_k(B)\longrightarrow\infty.
\end{equation}
If RH is false, every level is finite and nonempty, and the sets
eventually stabilize at the finite nonempty set $\mathcal M$.
\end{theorem}
\begin{proof}
Nesting follows from Theorem~\ref{thm:nesting}. A global maximizer
has a horizontal support at every level, so $\mathcal M\subseteq
B_k$; the reference is not such a maximizer because
$G_*<E=\lim_DG$.

Fix $n>n_*$ outside $\mathcal M$. If some $a<n$ in $D$ has
$G(a)\ge G(n)$, then $a$ strictly beats $n$ for every
$\lambda>0$, and $n$ has no horizontal support. Otherwise
$G(n)>G_*$, and a point $b>n$ exists with $G(b)>G(n)$.
First use the fixed ordinate $-G^{s_1}$. At abscissa $-q_k(n)$,
the coefficient of the reference in the chord from $n_*$ to $b$
is
\[
 \theta_k=\frac{q_k(n)-q_k(b)}{q_k(n_*)-q_k(b)}\longrightarrow0
\]
by \eqref{eq:weightescape}. The chord ordinate tends to
$-G(b)^{s_1}<-G(n)^{s_1}$, and eventually excludes $n$.
Changing $-G^{s_1}$ to $-G^{s_k}$ is increasing and concave,
and cannot restore a contact. This proves \eqref{eq:exact}.

Under RH, all values of $G$ are less than their limit $E$, so
the supremum is unattained. Under $\neg\RH$, some value exceeds
$E$; since $G(n)\to E$, the global maximum is attained at
finitely many integers and is greater than $E$. This proves the
first two equivalences in \eqref{eq:RHequivalences};
Proposition~\ref{prop:minima} proves the last. Maximum preservation
also proves the levelwise Robin property.

Under $\neg\RH$, let $b=\max\mathcal M$. For $n>b$, both
$G(n)<G(b)$ and $q_k(n)<q_k(b)$, so $b$ beats $n$ for every
$\lambda\ge0$. All levels lie in the same finite initial
interval, and nesting with \eqref{eq:exact} forces stabilization.
\end{proof}

The intersection is finite whenever it is nonempty. To prove RH
by \eqref{eq:RHequivalences}, one must prove escape without
assuming RH; nesting alone does not provide it.

\subsection{Examples with nonempty levels}
For any $w(n)>0$ strictly decreasing to zero, the power scheme
\begin{equation}\label{eq:powers}
 q_k(n)=w(n)^{r_k},\qquad 0<r_k\uparrow\infty,
 \qquad s_k>0\text{ nondecreasing}
\end{equation}
satisfies Theorem~\ref{thm:exact}. The transition of negative
abscissas is $x\mapsto-(-x)^{r_{k+1}/r_k}$. The natural choice
$r_k=s_k=k$ gives three examples:
\begin{equation}\label{eq:PLS}
 \begin{array}{c|ccc}
 \text{family}&P&L&S\\ \hline
 w(n)&u(n)^{-1}&[\log(1+u(n))]^{-1}&e^{-\sqrt{u(n)}}.
 \end{array}
\end{equation}
More generally, $e^{-u^\alpha}$ may replace the last weight for
any $0<\alpha<1$.

Another example is the fractional-linear scheme
\begin{equation}\label{eq:fractional}
 q_k(n)=\frac1{1+kt(n)},\qquad s_k=k,
\end{equation}
where $t$ increases strictly from $t(n_*)=0$ to infinity.
For $\ell\ge k$, the transition of negative abscissas is
\[
 x\longmapsto\frac{x}{a+(a-1)x},\qquad a=\ell/k,\quad -1\le x<0.
\]
Its first derivative is positive and its second derivative is
nonpositive; \eqref{eq:weightescape} also holds. Two choices are
$t=u/u_*-1$ and $t=\log(u/u_*)$.

\begin{proposition}[Infinitude and stabilization]\label{prop:infinite}
For every scheme in \eqref{eq:PLS}, for $w=e^{-u^\alpha}$ with
$0<\alpha<1$, and for the two fractional-linear choices above,
every level is nonempty unconditionally. Under RH every level is
infinite. Under $\neg\RH$ every level is finite, and the sets
eventually stabilize at $\mathcal M$. Thus infinitude of any
fixed level of any of these schemes is equivalent to RH.
\end{proposition}
\begin{proof}
Only the RH assertion remains. The Ramanujan CA estimates,
recalled in Section~\ref{sec:RN}, give, for fixed $k$,
\[
 0<\delta_k(n):=E^{s_k}-G(n)^{s_k}=O(e^{-u(n)/2}/u(n)).
\]
Each stated weight satisfies $\log(1/q_k)=o(u)$, hence
$\delta_k/q_k\to0$. For every $\lambda>0$ the function
$\delta_k-\lambda q_k$ is eventually negative and tends to zero,
so its negative minimum is attained. If contacts were confined to
a finite set, a sufficiently small $\lambda>0$ would make this
function positive at all of them and at the reference, a
contradiction. Thus there are infinitely many contacts.
\end{proof}

These schemes have infinitely many strict reductions under RH,
although strict inclusion at every step is not asserted. Their
infinitude concerns the full set and cannot be read from a finite
hull alone.

\subsection{First-contact records and quantitative extremality}
\begin{theorem}[Monotone Robin quotients]\label{thm:Gmin}
For a scheme satisfying Theorem~\ref{thm:exact} with every level
nonempty, $G(m_k(B))$ is nondecreasing, strictly increasing whenever
$m_k(B)$ increases, and
\begin{equation}\label{eq:Glimit}
 G(m_k(B))\longrightarrow\sup_{n\in D}G(n).
\end{equation}
RH is equivalent to Robin's inequality at every first contact,
and also to $\lim_kG(m_k(B))=E$.
\end{theorem}
\begin{proof}
If $n$ has a positive support parameter, comparison with any
$a<n$ in $D$ gives
\[
 G(n)^{s_k}-G(a)^{s_k}\ge
 \lambda(q_k(a)-q_k(n))>0.
\]
Thus $n$ is a strict record of $G$ on $D$. A contact having only
a horizontal support is a global maximizer. Every level contains
every global maximizer, so its least element cannot jump past
$\min\mathcal M$ when this exists. A new first contact with
positive support beats every earlier point; if it is a global
maximizer, the preceding first contact cannot already have been
one. This proves monotonicity and strictness at jumps.

If $m_k$ stabilizes, its eventual value is $\min\mathcal M$ by
\eqref{eq:exact}. If $m_k\to\infty$, then $\mathcal M$ is empty,
so every contact has positive support. Given $a\in D$, eventually
$m_k>a$ and $G(m_k)>G(a)$, proving \eqref{eq:Glimit}.
Under $\neg\RH$ the eventual first contact has value greater
than $E$; under RH every value is less than $E$ and the supremum
is $E$.
\end{proof}

For the arithmetic minima, by contrast, $G(m_k)\to E$ follows
unconditionally from Lemma~\ref{lem:CA-scale}. Monotonicity and
a Robin reduction to just their least members have not been proved.

For $P_k=\Cl_D(-u^{-k},-G^k)$, membership has an explicit
interpretation: for some $A>0$ and $B\ge0$,
\begin{equation}\label{eq:powercurve}
 \rho(m)^k\le A u(m)^k-B\quad(m\in D),\qquad
 \rho(n)^k=A u(n)^k-B.
\end{equation}
Indeed, $A=\max_D(G^k+B u^{-k})$. Orders $3,4,\ldots$
require global supporting curves $\rho^3=A u^3-B$,
$\rho^4=A u^4-B$, and so forth.

\begin{proposition}[Strengthening with the order]\label{prop:strength}
Normalize a support at $n$ by $b=B/\rho(n)^k$. For $1\le k<\ell$,
every normalized support $b\ge0$ at order $\ell$ is also a
normalized support at order $k$. If $b>0$, the order-$k$
inequality is strict for every $m\ne n$. Every $n\in P_k$
also satisfies the unconditional bound
\begin{equation}\label{eq:gapconstraint}
 G(n)^k\ge E^k-(E^k-G_*^k)(u_*/u(n))^k.
\end{equation}
If $G(n)<E$, then
\begin{equation}\label{eq:gapsimple}
 0<E-G(n)\le\frac E k
 \frac{(u_*/u(n))^k}{1-(u_*/u(n))^k}.
\end{equation}
Under RH the first contacts obey
\begin{equation}\label{eq:Pescape}
 \liminf_{k\to\infty}\frac{u(m_k(P))}{k\log k}\ge2.
\end{equation}
\end{proposition}
\begin{proof}
The normalized support is
\[
 (\rho(m)/\rho(n))^k+b\le(1+b)(u(m)/u(n))^k.
\]
Suppose it holds at order $\ell$. Put
$z=(u(m)/u(n))^\ell$, $v=(\rho(m)/\rho(n))^\ell$,
$t=(1+b)z-b\ge v>0$, and $r=k/\ell$. Concavity gives
$v^r\le t^r\le(1+b)z^r-b$. For $b>0$ and $m\ne n$ the
second inequality is strict.

A lower support at $(-u(n)^{-k},-G(n)^k)$ lies below the
reference and the limiting point $(0,-E^k)$, hence below their
chord. This gives \eqref{eq:gapconstraint}. Put $z=G(n)/E<1$
and $t=(u_*/u(n))^k$. Then $z^k\ge1-t$, whence
$1-z\le-\log z\le-\log(1-t)/k\le t/[k(1-t)]$.

Under RH, the Ramanujan lower gap bound gives $c>0$ such that
$E-G(n)\ge c e^{-u/2}/u$ for sufficiently large CA numbers.
Since $u_k=u(m_k(P))\to\infty$, \eqref{eq:gapsimple} yields
\[
 \frac{u_k}{2}\ge(k-1)\log u_k-k\log u_*+\log k+O(1).
\]
Here $(u_*/u_k)^k\to0$, so the omitted logarithmic correction
is bounded. Division by $k$ first gives $u_k/k\to\infty$;
then substituting $\log u_k\ge\log k$ gives \eqref{eq:Pescape}.
\end{proof}

For a fixed record $n$ that is not a global maximizer, one later
CA number $b$ with $G(b)>G(n)$ suffices to exclude $n$ from
all sufficiently high levels, by the proof of
Theorem~\ref{thm:exact}. The unresolved
arithmetic issue is to produce such a successor for every record
without assuming RH.

\section{Examples and comments}\label{sec:examples}

\subsection{Numerical first contacts}
Table~\ref{tab:numerical} gives first contacts
for arithmetic abscissas. Let $a_j$ be the increasing CA sequence
with $a_1=2$, so $a_9=n_*$ \cite{OEISca}. The table reports the index of
$m_k$, its number of decimal digits, and $u=\log\log m_k$.
Rounded values of $u$ are for display; the computations use
intervals rounded outwards with 70 decimal digits.

\begin{table}[htbp]
\centering\small
\caption{Arithmetic first contacts, without RH. Every displayed
index is checked against the entire infinite tail.}\label{tab:numerical}
\begin{tabular}{lrrrr}
\toprule
$F(u)$&$k$&CA index&Digits&$u(m_k)$\\
\midrule
$e^{\sqrt u}$&1&16&12&3.250465570\\
$e^{2\sqrt u}$&1&10&6&2.601800845\\
&2&16&12&3.250465570\\
&3&626&1834&8.348166205\\
\addlinespace
$e^{u^2}$&1--16&10&6&2.601800845\\
&32&27&25&4.024290044\\
&64&85&134&5.730804175\\
&128&490&1359&8.048426289\\
&200&2595&9775&10.021573915\\
\addlinespace
$e^{(\log u)^2}$&1--2&10&6&2.601800845\\
&3&16&12&3.250465570\\
&4&254&595&7.222014147\\
\bottomrule
\end{tabular}
\end{table}

Here $a_{10}=720720$ and $a_{16}=160626866400$. The plateau
$1\le k\le16$ for $e^{u^2}$ follows from the endpoint
values and nesting. At $k=200$ its first contact has $9775$
digits, while its double logarithm is close to
$\sqrt{200/2}=10$. For $e^{(\log u)^2}$ at $k=4$,
$\log u(m_4)=1.9771\ldots$, compared with the limiting scale $2$.

The critical-parameter algorithm independently checks the first
$10000$ prime-exponent events against the tabulated CA multipliers
\cite{OEISf}. All competing critical intervals are disjoint in
this range, so no tie is omitted. It compares the reference
secants \eqref{eq:secant} using outward intervals. For $F=e^f$,
the entire remaining tail is bounded using \eqref{eq:Robinupper}:
\begin{equation}\label{eq:tailbound}
 J_k(n)\le B_k(u):=
 \frac{(Eu+0.6483/u)^k}{e^{f(u)}-e^{f(u_*)}}.
\end{equation}
Differentiation gives
$\frac{d}{du}\log B_k(u)<k/u-f'(u)$.
For every displayed case, $u f'(u)>k$ holds at the cutoff
$u=11.5450072551\ldots$ and thereafter. The upper bound for
$B_k$ there is below the lower bound for the winning secant;
the largest tail-to-winner ratio is less than $0.948$.
Thus no uncomputed CA number can change a reported first contact.

\subsection{Ramanujan and Nicolas: terminating chains}\label{sec:RN}
The normalized Robin gap and its natural scale are
\begin{equation}\label{eq:Rgap}
 h(n)=\frac{e^{-u(n)/2}}{u(n)},\qquad
 \Rgap(n)=\frac{E-G(n)}{h(n)}
       =\sqrt{\log n}\,[Eu(n)-\rho(n)].
\end{equation}
Ramanujan's CA estimates, with the strong lower extension in
\cite{Musin}, imply under RH that
\begin{equation}\label{eq:Ramanujan}
 c_1\le\liminf_{n\in D}\Rgap(n)
 \le\limsup_{n\in D}\Rgap(n)\le c_2,
\end{equation}
where
\begin{align*}
 c_1&=E(2\sqrt2-4-\gamma+\log(4\pi))=1.3932184417\ldots,\\
 c_2&=E(2\sqrt2+\gamma-\log(4\pi))=1.5577589626\ldots.
\end{align*}
For every $0<c<c_1$, the strengthened inequality
\begin{equation}\label{eq:Ramstrength}
 \sigma(n)+\frac{cn}{\sqrt{\log n}}<En\log\log n
\end{equation}
holds eventually under RH (the strong lower extension allows all
integers here). A limit statement gives no explicit starting point
and does not justify setting $c=c_1$.

Nicolas \cite[Theorem 1.1 and Corollary 1.2]{Nicolas} gives
effective forms, including a bound involving the nontrivial zeta
zeros. In our normalization his Corollary 1.2 implies under RH
\begin{equation}\label{eq:Nicolas}
 \Rgap(n)\ge\Rgap(110880)=0.3863541276\ldots
 \qquad(n>55440),
\end{equation}
which exceeds $R_*:=\Rgap(55440)=0.2356799944\ldots$.
These are stronger pointwise inequalities than Robin's on the
stated ranges. Their use imports additional gap information,
rather than making the proof of RH easier. If RH is proved, the
resulting empty classes and strengthened inequalities become
unconditional properties.

\begin{proposition}[An exact chord test]\label{prop:chordempty}
For $r,s>0$, put $H_{r,s}=\Cl_D(-h^r,-G^s)$ and
$\delta_s(n)=E^s-G(n)^s$. Then $H_{r,s}=\varnothing$ if and
only if, for every $n\in D\setminus\{n_*\}$,
\begin{equation}\label{eq:chordempty}
 \delta_s(n)>\delta_s(n_*)[h(n)/h(n_*)]^r.
\end{equation}
Under RH, $H_{r,s}$ is empty for every $r\ge1$ and $s>0$.
Thus $\Cl_D(-h^k,-G^k)$ terminates at its first level under
RH, whereas under $\neg\RH$ it eventually stabilizes at
$\mathcal M$.
\end{proposition}
\begin{proof}
The reference and the limiting point $(0,-E^s)$ determine the
chord in \eqref{eq:chordempty}; points strictly above it cannot
have lower supports. Conversely, put
$\lambda_* =\delta_s(n_*)/h(n_*)^r>0$. If the inequality fails,
$\delta_s-\lambda_*h^r$ is nonpositive at a nonreference point,
is zero at the reference, and tends to zero at infinity. A
negative value forces an attained negative minimum. If no value
is negative, the nonreference zero is itself a minimum. Either
case gives a nonreference contact.

Under RH, if $G(n)\le G_*$, then
$\delta_s(n)\ge\delta_s(n_*)$, so \eqref{eq:chordempty} holds.
Otherwise put $z=G(n)/E>z_*=G_*/E$ and
$a_s(z)=(1-z^s)/(1-z)$. For $s\ge1$, $a_s$ is increasing,
so $a_s(z)/a_s(z_*)\ge1$. For $0<s<1$, the integral formula
for the secant slope gives $a_s(z)\ge s$ and
$a_s(z_*)\le s z_*^{s-1}\le s/z_*$. Hence in both cases
\[
 \frac{\delta_s(n)}{\delta_s(n_*)}
 \ge\frac{G_*}{E}\frac{E-G(n)}{E-G_*}
 \ge\frac{G_*}{E}\frac{\Rgap(110880)}{R_*}
           \frac{h(n)}{h(n_*)}
 >\frac{h(n)}{h(n_*)}.
\]
The last numerical factor exceeds $1.61$. Since $0<h(n)/h(n_*)<1$
and $r\ge1$, this proves \eqref{eq:chordempty}.
The last assertion follows from Theorem~\ref{thm:exact}.
\end{proof}

\begin{corollary}[A reference-normalized inequality]\label{cor:newineq}
Under RH, for every integer $n>55440$,
\begin{equation}\label{eq:newineq}
 \sigma(n)+R_*\frac n{\sqrt{\log n}}<En\log\log n.
\end{equation}
The assertion on the CA tail alone is equivalent to RH, and
also to $H_{1,1}=\varnothing$.
\end{corollary}
\begin{proof}
The strict bound follows from \eqref{eq:Nicolas} and
$\Rgap(110880)>R_*$. On the CA tail it implies Robin's
inequality, hence RH. The chord test for $r=s=1$ is precisely
$\Rgap(n)>R_*$.
\end{proof}

This is a weaker consequence of Nicolas's bound, normalized to
the reference for the geometric construction; it is not an
improvement of his estimate. The terminating example complements
the slowly decreasing weights whose levels remain infinite under RH.

\subsection{Comparison of the families and remaining questions}
Table~\ref{tab:families} summarizes the roles of the examples.
All definitions are unconditional, and all displayed families
preserve Robin's criterion at every level.

\begin{table}[htbp]
\centering\small
\caption{Intersection and size of the principal families.}\label{tab:families}
\begin{tabular}{>{\raggedright\arraybackslash}p{.22\linewidth}
>{\raggedright\arraybackslash}p{.18\linewidth}
>{\raggedright\arraybackslash}p{.23\linewidth}
>{\raggedright\arraybackslash}p{.23\linewidth}}
\toprule
Family&Intersection&Under RH&Under $\neg\RH$\\
\midrule
$A_k(a,\beta)$, $Q_k(a,\eta)$
 &$\varnothing$ always
 &Infinite levels; $m_k\to\infty$
 &Infinite levels; $m_k\to\infty$\\
\addlinespace
$P,L,S$; the two fractional-linear examples
 &$\mathcal M$
 &Infinite levels; $m_k\to\infty$
 &Finite levels; eventually $\mathcal M$\\
\addlinespace
$\Cl_D(-h^k,-G^k)$
 &$\mathcal M$
 &Every level empty
 &Finite nonempty levels; eventually $\mathcal M$\\
\bottomrule
\end{tabular}
\end{table}

Wu \cite{Wu} separates CA numbers into three subclasses according
to the positions of $\log n$, $P(n)$, and the next prime
$p^+(n)$. His middle class $\CA_2$ is characterized by
$P(n)<\log n<p^+(n)$, and Robin's inequality can be reduced to
this class. The proof compares omitted CA numbers with other CA
numbers having larger Robin quotient. It is an arithmetic
reduction related to preservation of extremal values, but does
not identify $\CA_2$ with a level of our coordinate schemes or
provide a nested iteration of the same construction. We use it
as a comparison rather than as a step of the present method.

Record conditions for $G$ also occur for extremely abundant
numbers \cite{NY}. Our reference is $55440$, whereas that
definition uses $10080$; a supporting-line condition is also
stronger than being a record. Thus the sets should not be
identified from their common record property alone.

For the arithmetic families, can the full levelwise Robin tests
be compressed to first contacts? For example, one may ask whether
\begin{equation}\label{eq:open-first}
 [G(m_k(A(1)))<E\text{ for every }k]\quad\Longrightarrow\quad\RH.
\end{equation}
The reverse implication is immediate, but
Theorem~\ref{thm:transfer} does not prove the displayed implication:
its counterexamples need not be first contacts. This remains an
open question here. For maximum-preserving families, can an
improving successor of every supported record be constructed
arithmetically, or can escape of $m_k$ be proved in another way?
The intersection theorem identifies this as the remaining RH
problem for those families. Independently of either question,
the growth, divisibility, and exponent-window results provide
an unconditional theory of higher-order CA numbers.

\section{Consequences of a zero-free half-plane}\label{sec:zero-free}

For the natural arithmetic family $A_k(1)$, Section~\ref{sec:arithmetic}
proved $\log\log m_k\sim k$ and $m_k\mid m_{k+1}$.
We now obtain error terms for the location of $m_k$ and estimate
the prime factors of the quotient $m_{k+1}/m_k$.
The argument has two steps. A bound on the real parts of the
zeta zeros gives estimates for primes, and hence for the size
and divisor sum of every CA number. The exact classical
parameter of a first contact then relates these estimates to
its order $k$. This argument is independent of Robin's criterion.

\subsection{The zero-free hypothesis}
Let
\begin{equation}\label{eq:Theta}
 \Theta=\sup\{\Re\varrho:\zeta(\varrho)=0,\ 0<\Re\varrho<1\}.
\end{equation}
Thus $\Theta$ is the rightmost possible real part of a nontrivial
zero, and $1/2\le\Theta\le1$. By the functional equation, the
bound $\Theta\le b$ places all nontrivial zeros in the closed
strip $1-b\le\Re s\le b$.
The following is the zeta-function case of the statement of
Theorem~1.1 in the OpenAI preprint \cite{OpenAI}.

\begin{theorem}[OpenAI, {\cite[Theorem 1.1]{OpenAI}}, zeta-function case]
\label{thm:OpenAI}
The Riemann zeta function has no zero in $\Re s>7/8$.
Consequently,
\begin{equation}\label{eq:seven-eighths-input}
 \Theta\le\frac78,
\end{equation}
and every nontrivial zero lies in $1/8\le\Re s\le7/8$.
\end{theorem}

We cite the preprint version released on September~30, 2026.
The accompanying repository describes its manuscripts as being
at different stages of verification. Theorem~\ref{thm:OpenAI}
is used here as an external input. To make this dependence
precise, all the arguments below are proved under the explicit
hypothesis $\Theta\le b<1$. Only the final specialization
$b=7/8$ uses the cited preprint; it does not assume RH.

For the estimates below, fix
\begin{equation}\label{eq:strip-hypothesis}
 \frac12\le b<1,\qquad \Theta\le b,\qquad \delta=1-b.
\end{equation}
The positive number $\delta$ measures the distance of the strip
from the line $\Re s=1$. Constants in $O$-terms may depend on
the fixed value of $b$. Write $t(n)=\log n$, so that
$u(n)=\log t(n)$. We also use
\[
 \vartheta(x)=\sum_{p\le x}\log p,\qquad
 \psi(x)=\sum_{p^j\le x}\log p.
\]
The classical explicit formula gives
\begin{equation}\label{eq:strip-prime-error}
 \psi(x)=x+O(x^b\log^2x),\qquad
 \vartheta(x)=x+O(x^b\log^2x);
\end{equation}
see \cite[Chapters 13 and 15]{MV}. The second estimate also
uses $\psi(x)-\vartheta(x)=O(\sqrt x)$.

\subsection{From primes to CA numbers}
The next lemma gives estimates valid for every sufficiently
large CA number, with constants independent of that number.
This uniformity allows us to apply them later to entire levels
of the hierarchy, while $k$ tends to infinity.

\begin{lemma}[CA estimates under a zero-free hypothesis]\label{lem:strip-CA}
Under \eqref{eq:strip-hypothesis}, uniformly as $n\to\infty$
through CA numbers,
\begin{align}
 P(n)&=t(n)+O\bigl(t(n)^b(\log t(n))^2\bigr),\label{eq:strip-P}\\
 \log P(n)&=u(n)+O\bigl(u(n)^2e^{-\delta u(n)}\bigr),\label{eq:strip-logP}\\
 \rho(n)&=Eu(n)+O\bigl(u(n)^2e^{-\delta u(n)}\bigr),\label{eq:strip-rho}\\
 G(n)&=E+O\bigl(u(n)e^{-\delta u(n)}\bigr).\label{eq:strip-G}
\end{align}
If $\varepsilon$ is any classical maximizing parameter for $n$
and $z>1$ is defined by
$\varepsilon=\Phi(z,1)=\log(1+1/z)/\log z$, then
\begin{equation}\label{eq:strip-z}
 t(n)=z+O(z^b\log^2z).
\end{equation}
\end{lemma}
\begin{proof}
First we compare $\log n$ with its largest prime factor and
with the real cutoff $z$ in the classical parameter rule.
The exponent-layer estimate used in Lemma~\ref{lem:CA-scale}
also gives
\begin{equation}\label{eq:CA-core}
 0\le t(n)-\vartheta(P(n))\ll\sqrt{P(n)},\qquad
 t(n)=\vartheta(z)+O(\sqrt z).
\end{equation}
Indeed, the layers of exponent at least two have total logarithmic
weight $O(\sqrt z)$ by \cite[Lemma 3]{CNS}. Since
$P(n)\le z\le p^+(n)$, the difference
$\vartheta(z)-\vartheta(P(n))$ is at most $\log p^+(n)$.
A possible excluded first increment at the endpoint $z=p^+(n)$
changes the second formula by at most $\log z$. Also $z\asymp P(n)$.
Combining \eqref{eq:CA-core} with \eqref{eq:strip-prime-error}
and $P(n)\sim t(n)$ proves \eqref{eq:strip-P},
\eqref{eq:strip-logP}, and \eqref{eq:strip-z}.

To estimate $\rho(n)$, first consider the full Euler product
over primes up to $x$. Partial summation of
\eqref{eq:strip-prime-error} gives
\[
 \sum_{p\le x}\frac1p
 =\log\log x+B_1+O(x^{b-1}\log x).
\]
After subtracting the limiting constant, the remainder consists
of a boundary term of this size and an integral bounded by
$O(\int_x^\infty v^{b-2}\log v\,dv)$.
Expanding the logarithm of the Euler product and using Mertens'
constant identity therefore yields
\begin{equation}\label{eq:strip-Mertens}
 M(x):=\prod_{p\le x}(1-p^{-1})^{-1}
 =E\log x+O(x^{b-1}\log^2x).
\end{equation}

The finite exponents of a CA number introduce a smaller error.
Put $P=P(n)$ and $a_p=v_p(n)$. The first omitted increment
satisfies $\Phi(p,a_p+1)\le\varepsilon$. Since
$\Phi(p,j)\ge p^{-j}/(4\log p)$, we have
$p^{-a_p-1}\le4\varepsilon\log p$; also
$p^{-a_p-1}\le p^{-2}$ for $p\le P$.
Using $\varepsilon\asymp(P\log P)^{-1}$ and splitting at
$\sqrt P$, we obtain
\[
 \sum_{p\le P}p^{-a_p-1}
 \ll\varepsilon\vartheta(\sqrt P)+\sum_{m>\sqrt P}m^{-2}
 \ll P^{-1/2}.
\]
Consequently,
\[
 \rho(n)=M(P)\prod_{p\le P}(1-p^{-a_p-1})
 =M(P)+O(P^{-1/2}\log P).
\]
As $b\ge1/2$, equations \eqref{eq:strip-Mertens} and
\eqref{eq:strip-logP} imply \eqref{eq:strip-rho}. Division by
$u(n)$ proves \eqref{eq:strip-G}.
\end{proof}

\subsection{Locating the first contacts}
In the rest of this section put
\[
 A_k=A_k(1)=\Cl_D(\log n,-\rho(n)^k),\qquad
 m_k=\min A_k,\qquad t_* =\log n_*.
\]
Let $J_k$ and $\Lambda_k$ be as in \eqref{eq:secant}, with
$F(u)=e^u$. Thus
\begin{equation}\label{eq:natural-secant}
 J_k(n)=\frac{\rho(n)^k-\rho_*^k}{t(n)-t_*},\qquad
 \Lambda_k=\max_{n\in D\setminus\{n_*\}}J_k(n).
\end{equation}
Lemma~\ref{lem:secant} says that $m_k$ is the least maximizer
of $J_k$. The approximate expression
$J_k(n)\approx E^k u(n)^k e^{-u(n)}$ has its maximum at
$u=k$. The theorem below makes this approximation quantitative.
Its proof uses the exact parameter of $m_k$, so it also accounts
for the discreteness of the CA numbers.

\begin{theorem}[Location and uniform estimates]\label{thm:strip-first}
Assume \eqref{eq:strip-hypothesis}. As $k\to\infty$,
\begin{align}
 u(m_k)&=k+O(k^3e^{-\delta k}),\label{eq:strip-first-u}\\
 \log P(m_k)&=k+O(k^3e^{-\delta k}),\label{eq:strip-first-P}\\
 \sup_{n\in A_k}|G(n)-E|&=O(ke^{-\delta k}),\label{eq:strip-level-G}\\
 \Lambda_k^{1/k}&=\frac{Ek}{e}\bigl(1+O(ke^{-\delta k})\bigr).
 \label{eq:strip-first-Lambda}
\end{align}
For every fixed prime $p$,
\begin{equation}\label{eq:strip-first-exponent}
 v_p(m_k)=\frac{k+\log k}{\log p}+O_p(1).
\end{equation}
\end{theorem}
\begin{proof}
By Lemma~\ref{lem:secant}, the first contact $m_k$ has support
parameter $\lambda=\Lambda_k$.
Apply Theorem~\ref{thm:windows} with $F(u)=e^u$, so that
$f(x)=x$ in its proof. It follows that $m_k$ is the unique
classical CA maximizer at the parameter
\begin{equation}\label{eq:first-exact-parameter}
 \varepsilon_k=\frac{\Lambda_k}{k\rho(m_k)^k}
 =\frac{1-q_k}{k(t(m_k)-t_*)},\qquad
 q_k=\left(\frac{\rho_*}{\rho(m_k)}\right)^k.
\end{equation}
The factor $q_k$ is the contribution of the fixed reference
number. The identity is exact for each $k$.

Write $t=t(m_k)$ and $u=\log t$. Theorem~\ref{thm:rates}
gives $u\sim k$, and Lemma~\ref{lem:CA-scale} then implies
$q_k\le(C/k)^k$ for a fixed $C$.
Define $z=z_k$ by $\Phi(z,1)=\varepsilon_k$ and put $L=\log z$.
Lemma~\ref{lem:strip-CA} gives
\[
 t=z+O(z^bL^2),\qquad
 u=L+O(L^2e^{-\delta L}),\qquad L\sim k.
\]
The exact parameter identity implies
\begin{equation}\label{eq:first-k-identity}
 k=\frac{1-q_k}{(t-t_*)\Phi(z,1)}.
\end{equation}
Since
$\Phi(z,1)^{-1}=zL(1+O(z^{-1}))$, we obtain
\[
 k=L+O(L^3e^{-\delta L})+O(Lq_k).
\]
Both errors tend to zero because $L\sim k$ and $\delta>0$.
Hence $L=k+o(1)$, and substitution
back into this estimate gives
$L=k+O(k^3e^{-\delta k})$.
The estimate for $u-L$ now proves \eqref{eq:strip-first-u};
\eqref{eq:strip-logP} proves \eqref{eq:strip-first-P}.
Every $n\in A_k$ has $u(n)\ge u(m_k)=k+o(1)$, and
$u e^{-\delta u}$ is eventually decreasing. Thus the uniform
bound \eqref{eq:strip-G} proves \eqref{eq:strip-level-G}.

For the secant, equations \eqref{eq:natural-secant} and
\eqref{eq:strip-rho} give
\begin{align*}
 \frac1k\log\Lambda_k
 &=\log\rho(m_k)-\frac uk
   +\frac1k\log(1-q_k)
   -\frac1k\log(1-t_*e^{-u})\\
 &=\log E+\log u-\frac uk+O(ke^{-\delta k}).
\end{align*}
The function $h_k(u)=\log u-u/k$ is stationary at $u=k$.
By \eqref{eq:strip-first-u},
\[
 h_k(u)=\log k-1+O\bigl((u-k)^2/k^2\bigr)
       =\log k-1+O(ke^{-\delta k}),
\]
which proves \eqref{eq:strip-first-Lambda}.
Finally, \eqref{eq:fixedprime} and $u(m_k)=k+o(1)$ give
\eqref{eq:strip-first-exponent}.
\end{proof}

In particular, the difference $\log\log m_k-k$ tends to zero
exponentially, up to the factor $k^3$; the earlier asymptotic
$\log\log m_k\sim k$ only gave an error $o(k)$.
Equation~\eqref{eq:strip-level-G} controls every member of
$A_k(1)$, not only its first contact. The bounded remainder in
\eqref{eq:strip-first-exponent} allows for the integer rounding
of each prime exponent.

\subsection{Prime factors of successive quotients}
By Corollary~\ref{cor:effective}, $m_k\mid m_{k+1}$ for every
$k\ge1$. Put $r_k=m_{k+1}/m_k$. For an integer $r\ge1$,
let $\omega(r)$ be the number of distinct prime divisors and
$\Omega(r)$ the number of prime factors counted with multiplicity.
Write $\Li(x)=\int_2^x dv/\log v$ for $x>2$.

\begin{corollary}[Size and prime-factor counts]\label{cor:strip-quotients}
Under \eqref{eq:strip-hypothesis}, as $k\to\infty$,
\begin{align}
 \log m_k&=e^k+O(k^3e^{bk}),\label{eq:strip-logm}\\
 \log r_k&=(e-1)e^k+O(k^3e^{bk}),\label{eq:strip-logr}\\
 \nu(r_k)&=\Li(e^{k+1})-\Li(e^k)+O(k^2e^{bk})
 \qquad(\nu\in\{\omega,\Omega\}).\label{eq:strip-prime-count}
\end{align}
In particular,
\begin{equation}\label{eq:strip-prime-asymptotic}
 \omega(r_k)\sim\Omega(r_k)\sim(e-1)\frac{e^k}{k}.
\end{equation}
\end{corollary}
\begin{proof}
Exponentiating \eqref{eq:strip-first-u} proves
\eqref{eq:strip-logm}; subtracting consecutive instances gives
\eqref{eq:strip-logr}.
Put $P_k=P(m_k)$. Every prime in $(P_k,P_{k+1}]$ divides
$r_k$. The number of further prime factors, counted with
multiplicity, is at most
\[
 \sum_{p\mid m_{k+1}}(v_p(m_{k+1})-1)
 \le\frac{\log m_{k+1}-\vartheta(P_{k+1})}{\log2}
 =O(\sqrt{P_{k+1}})
\]
by \eqref{eq:CA-core}. Therefore, for either counting function,
\begin{equation}\label{eq:strip-nu-pi}
 \nu(r_k)=\pi(P_{k+1})-\pi(P_k)+O(\sqrt{P_{k+1}}).
\end{equation}
Partial summation of \eqref{eq:strip-prime-error} gives
$\pi(x)=\Li(x)+O(x^b\log x)$.
Also \eqref{eq:strip-first-P} implies
\[
 P_k=e^k+O(k^3e^{bk}),\qquad
 \Li(P_k)=\Li(e^k)+O(k^2e^{bk}).
\]
Apply these estimates at $k$ and $k+1$ in
\eqref{eq:strip-nu-pi}. Since $b\ge1/2$, the remaining error
is absorbed by $O(k^2e^{bk})$, proving
\eqref{eq:strip-prime-count}.
Finally,
\[
 \Li(e^{k+1})-\Li(e^k)=\int_k^{k+1}\frac{e^v}{v}\,dv
 \sim(e-1)e^k/k,
\]
whereas $k^2e^{bk}=o(e^k/k)$.
\end{proof}

The leading contribution to both $\omega(r_k)$ and $\Omega(r_k)$
comes from the new primes between $P(m_k)$ and $P(m_{k+1})$.
Increases in the exponents of smaller primes contribute only
to the error term. This explains why the two counting functions
have the same asymptotic formula.

\subsection{Consequences of \texorpdfstring{$\Theta\le7/8$}{Theta at most 7/8}}
Theorem~\ref{thm:OpenAI} supplies the fixed value $\delta=1/8$.
Substitution in the preceding results gives the following
estimates for the natural arithmetic family.

\begin{corollary}\label{cor:seven-eighths}
Assuming the statement of Theorem~\ref{thm:OpenAI}, the family
$A_k(1)$ satisfies, as $k\to\infty$,
\begin{align}
 \log\log m_k&=k+O(k^3e^{-k/8}),\label{eq:78-minimum}\\
 \log P(m_k)&=k+O(k^3e^{-k/8}),\label{eq:78-prime}\\
 \sup_{n\in A_k(1)}|G(n)-e^\gamma|&=O(ke^{-k/8}),\label{eq:78-level}\\
 \Lambda_k^{1/k}&=e^{\gamma-1}k\bigl(1+O(ke^{-k/8})\bigr),\label{eq:78-secant}\\
 \log r_k&=(e-1)e^k+O(k^3e^{7k/8}),\label{eq:78-quotient}\\
 \nu(r_k)&=\Li(e^{k+1})-\Li(e^k)+O(k^2e^{7k/8})
 \quad(\nu\in\{\omega,\Omega\}).\label{eq:78-count}
\end{align}
Equations \eqref{eq:strip-first-exponent} and
\eqref{eq:strip-prime-asymptotic} hold as well.
\end{corollary}
\begin{proof}
Take $b=7/8$ in Theorem~\ref{thm:strip-first} and
Corollary~\ref{cor:strip-quotients}.
\end{proof}

For example, \eqref{eq:78-minimum} is equivalently
\[
 \log m_k=e^k\bigl(1+O(k^3e^{-k/8})\bigr).
\]
Thus the size of the first contact has a relative error tending
to zero at this explicit rate on the logarithmic scale.
Similarly, the error in \eqref{eq:78-count} is smaller than
its main term by a factor $O(k^3e^{-k/8})$.
The constants and starting values in these asymptotic estimates
are not made numerical here.

The estimate \eqref{eq:78-level} bounds the absolute deviation
of $G$ from $e^\gamma$, without determining its sign. The
Ramanujan--Nicolas results in Section~\ref{sec:RN} give stronger,
signed gap estimates under RH. The present conclusions concern
the location and factorization of higher-order CA numbers and
the uniform approximation of $G$ on their levels. Their
specialization at $7/8$ requires the cited preprint result,
but no assumption of RH.

\paragraph{Acknowledgments.}
The author used ChatGPT (OpenAI, November 2026) for assistance
in deriving the estimates in Section~\ref{sec:zero-free} and
preparing the text. The author is responsible for the arguments
and the final version.

\bigskip
\noindent\textsc{Oleg R. Musin}\\
School of Mathematical and Statistical Sciences\\
The University of Texas Rio Grande Valley\\
One West University Boulevard, Brownsville, TX 78520, USA\\
\textit{E-mail:} \href{mailto:omusin@gmail.com}{omusin@gmail.com}
\end{document}